\documentclass[11pt,letterpaper]{amsart}

\usepackage[utf8]{inputenc}
\usepackage[english]{babel}
\usepackage{amsmath}
\usepackage{amsfonts}
\usepackage{amssymb}
\usepackage{epstopdf}
\usepackage[all]{xy}
\usepackage{amsmath,amsthm}
\usepackage{amsmath,amssymb}
\usepackage{amsmath,chemarrow}
\usepackage{appendix}

\usepackage{latexsym}
\usepackage{tikz}
\usepackage{tikz-cd}
\usepackage{amsmath,amsthm}

\usepackage[mathscr]{eucal}
\usepackage[all]{xy}
\usepackage{mathrsfs}
\usepackage{hyperref}
\usepackage{pifont}

\allowdisplaybreaks

\newtheorem{thm}{Theorem}[section]
\newtheorem{cor}[thm]{Corollary}
\newtheorem{lem}[thm]{Lemma}
\newtheorem{prop}[thm]{Proposition}
\theoremstyle{definition}
\newtheorem{defn}[thm]{Definition}
\newtheorem{rem}[thm]{Remark}
\newtheorem{examp}[thm]{Example}

\numberwithin{equation}{section}

\newcommand{\sD}{{\mathcal D}}
\newcommand{\sE}{{\mathcal E}}
\newcommand{\sF}{{\mathcal F}}
\newcommand{\sG}{{\mathcal G}}

\newcommand{\sK}{{\mathcal K}}
\newcommand{\sL}{{\mathcal L}}

\newcommand{\sO}{{\mathcal O}}

\newcommand{\sX}{{\mathcal X}}

\newcommand{\sZ}{{\mathcal Z}}
\newcommand{\A}{{\mathbb A}}

\newcommand{\N}{{\mathbb N}}

\newcommand{\Q}{{\mathbb Q}}

\newcommand{\R}{{\mathbb R}}

\newcommand{\Z}{{\mathbb Z}}

\newcommand{\Spec}{\mathrm{Spec}}

\newcommand{\rk}{\mathrm{rk}}
\newcommand{\pardeg}{\mathrm{pardeg}}

\newcommand{\Gr}{\mathrm{Gr}}

\newcommand{\chara}{\mathrm{char}}
\author[Mao Sheng]{Mao Sheng}
\author[Jianping Wang]{Jianping Wang}

\email{msheng@ustc.edu.cn}

\email{jianpw@ustc.edu.cn}

\address{School of Mathematical Sciences, University of Science and Technology of China, Hefei, 230026, China}

\begin{document}
	
	\title[Tensor product theorem]{Tensor product theorem for parabolic $\lambda$-connections}
	
	\begin{abstract}
		We establish a tensor product theorem for slope semistable parabolic $\lambda$-connections over smooth projective varieties in arbitrary characteristic.
	\end{abstract}

	\maketitle
	\section{Introduction}
	A basic result in the theory of vector bundles is the so-called \emph{tensor product theorem}: the tensor product of two slope semistable vector bundles over a smooth projective variety in char zero is again slope semistable. There are several different approaches to establish the theorem. The aim of the paper is to apply the recent theory of Higgs-de Rham flows \cite{LSZ19} to this type problem. Our result is the following generalization of the tensor product theorem to parabolic vector bundles equipped with (not necessarily integrable) $\lambda$-connections.    
	\begin{thm}\label{main thm}
		Let $k$ be an algebraically closed field. Let $X$ be a smooth projective variety over $k$ and $L$ an ample line bundle over $X$. Let $D$ be a reduced effective normal crossing divisor in $X$. Let $\lambda\in k$. Let $(E^i_{\bullet},\nabla^i), i=1,2$ be two $\mu_L$-semistable parabolic $\lambda$-connections over $(X,D)$. If either $\textrm{char}\ k=0$ \emph{or} $\textrm{char}\ k=p>0$ and $\rk(E^1)+\rk(E^2)\leq p+1$ , then their tensor product is $\mu_L$-semistable.
	\end{thm}
	Note that torsion freeness has been built into our definitions of parabolic sheaves and their tensor products (see \S2).  
	
	In char zero, for Higgs bundles over $X$, the above result is due to C. Simpson (\cite[Corollary 3.8] {S92}). He first reduced the theorem to the curve case and then used the existence of Higgs-Yang-Mills connections for stable subquotients in a Jordan-H\"older filtration attached to a semistable Higgs bundle. However, in this approach, one might not be able to weaken the regularity condition on $X$ (compare \cite[Theorem 3.1.4]{HL10}, \cite[Theorem 5.16]{M16}), since the restriction theorem of Mehta-Ramanathan \cite{MR82} requires $X$ be smooth. 
	
	Our method is basically to algebraize Simpson's approach. We shall use the reduction modulo $p$. Recall that in char $p$, Balaji-Parameswaran (\cite[Theorem 1.1]{BP11}) proved that the tensor product of two semistable Higgs bundles over a smooth projective curve with trivial determinant and the rank condition \footnote{The rank condition of Balaji-Parameswaran is same as the one given in Theorem \ref{main thm}. It is interesting to note that the condition shows up from a totally different reason. See Step 3-4 in the proof of Theorem \ref{main thm}.} is semistable. In this paper, we shall give it another proof, using the existence result on semistable Higgs-de Rham flows in char $p$ (\cite[Theorem A.1]{LSZ19}, \cite[Theorem 5.12]{LA13}). Our approach has some new features which are absent in the GIT approach: (i) the presence of $D$, that is, considering logarithmic Higgs fields does not introduce additional difficulty; (ii) the same method treats semistable connections on an equal footing. We remark that, even in the curve case, our result for connections cannot be directly deduced from the theorem of Balaji-Parameswaran. On the other hand, a Metha-Ramanathan type result for parabolic $\lambda$-connections should be supplied. In this aspect, one has the result of Simpson (\cite[Lemma 3.7]{S92}) for Higgs sheaves and the result of Mochizuki (\cite[Proposition 3.29]{MO06}) for parabolic Higgs sheaves. We observe that the integrality of Higgs fields plays no role in the proofs, and therefore the restriction theorem holds also for parabolic 0-connections. However, the argument does not work for connections \footnote{Indeed, for a smooth hypersurface $Y\subset X$, a parabolic $\lambda$-connection over $(X,D)$ induces a $\Omega_X(\log D+Y)|_Y$-valued operator by restriction. It is only an $\Omega_X(\log D)|_Y$-valued operator when $\lambda=0$.}. We circumvent this difficulty by exploiting again the idea of Simpson on the existence of a gr-semistable filtration on a semistable connection. 
	
	In many ways, the notion of a Higgs-de Rham flow is a good analogue of Higgs-Yang-Mills flow in complex algebraic geometry. We would like to remind our readers of the recent work of Arapura \cite{A19}, the work of Esnault-Groechenig \cite{EG}, the work of Langer \cite{LA15}-\cite{LA16} and the work of Lan-Sheng-Yang-Zuo \cite{LSYZ19}.

	\section{Preliminaries on parabolic $\lambda$-connections}\label{preliminaries}
	Fix a locally noetherian scheme $S$. Let $X$ be a smooth $S$-scheme, and $D=\sum\limits_{i=1}^{w}D_i\subset X$ be an effective relative reduced normal crossing $S$-divisor. For brevity, we write the pair $(X,D)$ by $X_{\log}$. Following \cite{MY92}, we make the following
	\begin{defn}
		Let $E$ be a torsion free coherent $\mathcal{O}_X$-module which is flat over $S$. A parabolic structure on $E$ with parabolic support in $D$ is an $\mathbb{R}$-indexed filtration of coherent sheaves $E_{\bullet}$ over $X$ such that
		\begin{enumerate}
			\item[(0)] $E_t$ is flat over $S$, for $t\in \R$,
			
			\item[(1)] $E_0=E$,
			
			\item[(2)] $E_{t} \subseteq E_{s}$, for $t \geq s$,
			
			\item[(3)] $\exists  \epsilon>0$ such that $E_{t-\epsilon} = E_{t}$ holds for all $t\in \R$.
			
			\item[(4)] $E_{t+1}=E_t(-D)$, for $t\in \R$.
		\end{enumerate}
		We call $E$ together with a parabolic structure $E_{\bullet}$ a \emph{parabolic sheaf} over $X_{\log}/S$.   	
	\end{defn}
	For an $1\leq i\leq w$, by (3) there are only finitely many numbers $\{a_{ij}\}\subset [0,1)$ such that $\Gr_{a_{ij}}(E_{\bullet}):=(\frac{E_{a_{ij}}}{E_{a_{ij}+\delta}})|_{D_i}$ is nonzero for small enough $\delta>0$. The collection $\{a_{ij}\}_{j}$ is called the set of \emph{parabolic weights} of $E_{\bullet}$ along $D_i$. A parabolic weight is said to be \emph{rational} if it is a ratoinal number. In this paper, the \emph{parabolic filtration} associated to $E_{\bullet}$ means the finite filtration 
	$$
	E=E_0\supsetneq E_1\supsetneq \cdots\supsetneq E_l\supsetneq E_{l+1}=E(-D),
	$$
	by forgetting the parabolic weights of $E_{\bullet}$. One defines a morphism of parabolic sheaves, a parabolic coherent subsheaf (quotient sheaf), and the direct sum of parabolic sheaves in the natural way. However they do not form an abelian category, because of the torsion free assumption.
	\begin{examp}\label{example}
		\begin{itemize}
			\item [(a)] Trivial parabolic structure. The trivial parabolic structure on $E$ is defined by $E_t:=E(-\left\lceil t \right\rceil D)$, where $\left\lceil \centerdot \right\rceil$ is the ceiling function. The parabolic weights are all zero. 
			\item [(b)]	Parabolic vector bundle. For an invertible sheaf $L$ over $X$ and a $w$-tuple $\vec{a}=(a_i)\in [0,1)^{w}$, one defines a parabolic structure on $L$ by
			$$L^{\vec{a}}_t:=L\Big(\sum_{1\leq i\leq w}-\left\lceil t-a_i \right\rceil D_i\Big).$$ The parabolic weight of $L^{\vec{a}}_{\bullet}$ along $D_i$ is $a_i$. A parabolic vector bundle over $X_{\log}$ is a vector bundle $V$ over $X$, equipped with a parabolic structure which Zariski locally is a direct sum of $L_{\bullet}^{\vec{a}}$s. This is called a \emph{locally abelian} parabolic bundle in \cite{IS}. 
		\end{itemize}
	\end{examp}
	\begin{lem}\label{parabolic vb over curve}
		For $S=\Spec\ k$ with $k$ algebraically closed and $\dim X=1$, any parabolic sheaf over $X_{\log}/k$ is a parabolic vector bundle. 	
	\end{lem}
	
	\begin{proof}
		Let $E_\bullet$ be a parabolic sheaf over $X_{\log}$. As $E$ is torsion free and $X$ is a smooth curve over a field, $E$ is locally free, viz. a vector bundle over $X$. Let $\{0=a_0<a_1<\cdots<a_l<1\}$ be the parabolic weights of $E_{\bullet}$ at a point $x\in D$. Then the parabolic structure $E_{\bullet}$ restricts to a flag of vector spaces over $k=k(x)$:
		$$
		E(x)=F_0(E)\supsetneq F_1(E_x) \supsetneq \dots \supsetneq F_l(E_x)\supsetneq F_{l+1}(E_x)=0.
		$$
		Set $n_i=\dim_k(F_i(E_x)/F_{i-1}(E_x))$. Choose a basis of $E(x)$ to split the above filtration. Then, by Nakayama lemma, it follows that locally around $x$, there is an isomorphism of parabolic sheaves:
		$$
		E_{\bullet}\cong \bigoplus\limits_{0\leq i\leq l}((\mathcal{O}_X)_{\bullet}^{a_i})^{\oplus n_i},
		$$
		where $(\mathcal{O}_X)_{\bullet}^{a_i}$ is the parabolic line bundle as defined in \ref{example}(b). 
	\end{proof}
	
	To define the tensor product of two parabolic sheaves, we set $j: U=X-D\to X$ to be the open immersion. For any coherent sheaf $E$ over $X$, put $E_{\infty}=\lim\limits_{\longrightarrow}E(nD)\subset j_*j^*E$. For a parabolic sheaf $E_{\bullet}$ over $X_{\log}$, there is a natural inclusion $E_t\to E_{\infty}$ for any $t$. Then for two parabolic sheaves $E_{\bullet}$ and $F_{\bullet}$ over $X_{\log}$, define for $t\in \R$
	$$
	(E_{\bullet}\otimes F_{\bullet})_t=\textrm{Im}(\bigoplus_{t_1+t_2=t}E_{t_1}\otimes F_{t_2}\longrightarrow E_{\infty}\otimes F_{\infty}/\textrm{torsion}).
	$$
	It is not difficult to verify that $(E_{\bullet}\otimes F_{\bullet})_{\bullet}$ defines a parabolic structure on $(E_{\bullet}\otimes F_{\bullet})_0$.
	
	Now we fix an element $\lambda\in \Gamma(S,\sO_S)$, and for a parabolic sheaf $E_{\bullet}$ over $X_{\log}/S$, we make the following
	\begin{defn}
		An $S$-relative $\lambda$-connection $\nabla$ on $E_{\bullet}$ is a collection $\{\nabla_t\}_{t\in \R}$ of $\sO_S$-linear maps with $\nabla_t: E_t\to E_t\otimes \Omega_{X/S}(\log D)$ satisfying
		\begin{itemize}
			\item [(a)] $\nabla(fs)=\lambda df\otimes s+f\nabla(s)$ for $f\in \sO_X$ and $s\in E_t$;
			\item [(b)] For any $t\geq s$, the following diagram commutes:
			\begin{equation*}
				\xymatrix{
					E_t\ar[r]^-{\nabla_t} \ar[d]_{\iota_{t,s}} & E_t\otimes \Omega_{X/S}(\log D) \ar[d]^{ \iota_{t,s}\otimes id}\\
					E_s\ar[r]^-{\nabla_s} & E_s\otimes \Omega_{X/S}(\log D)\\
				}
			\end{equation*}
			where $\iota_{t,s}: E_t\to E_s$ is the natural inclusion.
		\end{itemize}
	\end{defn}
	\begin{rem}\label{remark on lambda}
		A $\lambda$-connection is said to be integrable if $\nabla\wedge \nabla=0$. The integrability condition is automatic if $X/S$ is of relative dimension one.  For a $\lambda$-connection $\nabla$ on $E_{\bullet}$, there is a natural $\lambda$-connection on $E_{\infty}$. Therefore, for two parabolic $\lambda$-connections $(E_{\bullet},\nabla_E)$ and $(F_{\bullet},\nabla_F)$, their tensor product is defined to be the $(E_{\bullet}\otimes F_{\bullet})_{\bullet}$, equpped with the $\lambda$-connection induced from the tensor product of $\lambda$-connections over $E_{\infty}\otimes F_{\infty}$.
	\end{rem}
In order to handle parabolic structures in this paper, we need to recall the generalized Biswas-Iyer-Simpson (BIS) correspondence for parabolic $\lambda$-connections in all characteristics (\cite{Bi},\cite{IS},\cite[\S2]{KS20}). Let $k$ be an algebraically closed field and let $N$ be a positive integer such that $\chara\ k\nmid N$. Let $X$ be a smooth curve over $k$ and let $\pi: X'\to X$ be a cyclic cover of order $N$ with branch divisor $D\subset X$. The constant group scheme $G\cong \Z/N\Z$ acts on $X'$ with quotient $X'/G\cong X$. Let $D'=\sum_{i=1}^lD_i'$ be the reduced divisor $(\pi^*D)_{red}$. Set $X'_{\log}$ to be the log pair $(X',D')$. 
\begin{prop}\cite[Proposition 2.14]{KS20}\label{BIS}
Notations as above. The parabolic pushfoward $\pi_{\mathrm{par}*}$ and parabolic pullback $\pi_{\mathrm{par}}^*$ induce an equivalence of categories between the category of $G$-equivariant $\lambda$-connections over $X'_{\log}$ and the category of parabolic $\lambda$-connections over $X_{\log}$ whose parabolic structure is supported along $D$ and has weights in $\frac{1}{N}\Z$.
\end{prop}
When $X/k$ projective, equipped with an ample line bundle $L$, one may define the \emph{parabolic degree} of a parabolic sheaf $E_{\bullet}$ over $X_{\log}/k$ by the following formula:
	$$
	\pardeg_L(E_{\bullet})=\deg_{L}(E_{0})+\sum_{1\leq i\leq w}\sum_ja_{ij}n_{ij}\deg_LD_i,
	$$
	where $n_{ij}:=\rk \Gr_{a_{ij}}(E_{\bullet})$. Clearly, for a parabolic sheaf with the trivial parabolic structure, the parabolic degree and the usual degree coincide.
	\begin{defn}
		Let $k$ be an algebraically closed field. Let $(E_{\bullet},\nabla)$ be a $\lambda$-connection over $X_{\log}$. It is said to be $\mu_L$-(semi)stable, if for any $\nabla$-invariant parabolic coherent subsheaf $F_{\bullet}\subset E_{\bullet}$, the following inequality holds:
		$$
		\rk(E)\pardeg_L(F_{\bullet})(\leq)<\rk(F)\pardeg_L(E_{\bullet}).
		$$
	\end{defn}
	For a nonzero parabolic sheaf $E_{\bullet}$ over $X_{\log}$, one defines its parabolic slope (with respect to $L$) by $$\mu_L(E_{\bullet})=\pardeg_L(E_{\bullet})/\rk(E).$$ 
	\begin{lem}\label{tensorslope}
		Notations as above. Let $F^i_\bullet,\ i=1,2$ be two parabolic sheaves over $X_{\log}$. Then $$\mu_L(F^1_\bullet\otimes F^2_\bullet)=\mu_L(F^1_\bullet)+\mu_L(F^2_\bullet).$$
	\end{lem}
	\begin{proof}
		We may assume $L$ to be very ample. By Bertini's theorem (\cite[Theorem 8.18]{H}), we may restrict $F^i_{\bullet}$ to a smooth complete intersection curve	in the linear system $|L|$, and $\pardeg(F^i_{\bullet})$ equals the parabolic degree of the restricted parabolic sheaf. Therefore, we may assume $\dim X=1$. By Lemma \ref{parabolic vb over curve}, we may assume then $F^i_{\bullet}$s are parabolic vector bundles.
		
		Consider first the case that the parabolic structures on $F^{i}_{\bullet}$s are trivial. The truth of this case is a standard fact: one may use the splitting principle to reduce the problem to the case of line bundles that is obviously true. Consider further the case that one of the parabolic structures is nontrivial but satisfies either i) $\chara\ k=0$, the parabolic weights are rational, or ii) $\chara\ k=p>0$, the parabolic weights are rational whose denominators are coprime to $p$. One applies Proposition \ref{BIS} to reduce this case to the previous one. For the general case, we use a density argument as follows: let 
		$$
		V^i=\{\vec{a^i}=(a^i_{jk})|a^i_{jk}\in \R, 0\leq a^i_{jk}<1\}, i=1,2
		$$ 
		be the weight space containing the parabolic weights (say $\vec{a_0^i}$) of $F^i_\bullet$. For $\vec{a^i}\in V^i$, we denote by $F^{i,\vec{a^i}}_\bullet$ the parabolic bundle whose associated parabolic filtration is the same as that of $F^i_{\bullet}$ but with parabolic weights $\vec{a^i}$. In this terminology, $F^{i}_\bullet=F^{i,\vec{a_0^i}}_\bullet$. Define the function
		$$
		\Psi: V^1\times V^2 \longrightarrow \R,\quad (\vec{a^1},\vec{a^2})\longmapsto \mu(F^{1,\vec{a^1}}_\bullet\otimes F^{2,\vec{a^2}}_\bullet)-\mu(F^{1,\vec{a^1}}_\bullet)-\mu(F^{2,\vec{a^2}}_\bullet).
		$$
		By definition of parabolic slope, it is continuous. Moreover, we notice that the subset $Z\subset V^1\times V^2$, consisting of the parabolic weights in the second case, is actually dense. Since by the second case $\Psi$ restricts to $Z$ as the zero map, it is identically zero. Thus the general case follows.

	\end{proof}

	\section{Restriction theorem}
	Mehta and Ramanathan \cite[Theorem 6.1]{MR82} proved that the restriction of a semistable torsion free coherent sheaf to a general hyperplane section of sufficiently high degree is again semistable. Simpson \cite[Lemma 3.7]{S92} generalized it to Higgs sheaves, and later Mochizuiki \cite[Proposition 3.29]{MO06} generalized it further to parabolic Higgs sheaves. Langer \cite[Theorem 10]{LA15} has obtained a much stronger restriction theorem, in that the degree of a general hypersurface can be made explicit. To our purpose, it suffices to generalize the Mehta-Ramanathan restriction theorem to the semistable parabolic $\lambda$-connections. 	
	
	We resume the setting of Theorem \ref{main thm}. Let $(E_\bullet,\nabla)$ be a $\mu_L$-semistable parabolic $\lambda$-connection over $X_{\mathrm{log}}$. By Bertini's theorem, for a general member $Y\in |\sO_X(m)|$ for $m$ large enough, $E_\bullet$ restricts to a parabolic sheaf $E_{\bullet}|_Y$ over $Y_{\log}=(Y,Y\cap D)$, and $Y+D$ is normal crossing so that $D\cap Y$ is a reduced effective normal crossing divisor on $Y$. By a simple local calculation, one sees that the parabolic subsheaf $E_{\bullet}\otimes \sO_{X}(-Y)\subset E_{\bullet}$ is annihilated by the composite map
	$$
	E_{\bullet}\stackrel{\nabla}{\longrightarrow} E_{\bullet}\otimes \Omega_{X/k}(\log D)\stackrel{\textrm{res}}{\longrightarrow}  E_{\bullet}|_{Y}\otimes\Omega_{X/k}(\log D)|_Y \stackrel{
		\textrm{id}\otimes \textrm{ev}}{\longrightarrow} E_{\bullet}|_Y\otimes \Omega_{Y/k}(\log D\cap Y). 
	$$ 
	So $\nabla$ on $E_{\bullet}$ yields the operator $\nabla|_{Y}$ on $E_{\bullet}|_Y$, which is easily seen to be a parabolic $\lambda$-connection over $(Y,D\cap Y)/k$. The aim of the section is the following theorem.
	\begin{thm}\label{restriction thm}
		Let $(X_{\mathrm{log}},L,E_\bullet,\nabla)$ be as above. Then for general hyperplane sections $Y$ of sufficiently high degree, $(E_{\bullet},\nabla)|_{Y}:=(E_{\bullet}|_{Y},\nabla|_{Y})$ is a semistable parabolic $\lambda$-connection.
	\end{thm}

	Let us consider an arbitrary parabolic $\lambda$-connection $(E_\bullet,\nabla)$ over $X_{\log}/k$. A level $n$ filtration $Fil=Fil^*_{\bullet}$ of parabolic subsheaves of $E_{\bullet}$
	$$
	E_\bullet=Fil_\bullet^0 \supset Fil_\bullet^1 \supset \dots \supset Fil_\bullet^n\supset 0
	$$ 
	is said to be \emph{Griffiths transverse} if 
	$$
	\nabla(Fil_\bullet^*) \subset Fil_\bullet^{*-1} \otimes \Omega_{X/k}(\log D)
	$$ 
	holds. Taking grading yields a collection $\{Gr_{Fil_{t}}(E_{t},\nabla)\}_{t\in \R}$ of logarithmic 0-connections and morphisms $\{\phi_{t,s}\}_{t\geq s}$ of logarithmic 0-connections
	$$
	\phi_{t,s}: Gr_{Fil_{t}}(E_{t},\nabla)\to Gr_{Fil_{s}}(E_{s},\nabla).
	$$  
	We consider a construction of $Fil_{\bullet}$ such that $\phi_{t,s}$s are all injective. 
	\begin{lem}\label{induced filtration}
		Let $Fil_0$ be a saturated and Griffiths transverse filtration on $(E_0,\nabla)$. For $n\in \Z$, set $Fil_{n}=Fil_0\otimes Fil_{tr}$ to be the tensor filtration over $V_n\cong V_0\otimes \sO_{X}(-nD)$, where $Fil_{tr}$ stands for the trivial filtration. For $n\leq t<n+1$, set $Fil_t=E_t\cap Fil_{n}$ to be the restricted filtration over $E_t$. Then, the filtration $Fil=Fil_{\bullet}$ defined as above is a Griffiths transverse filtration with $\phi_{t,s}$ injective for any $t\geq s$.  
	\end{lem}
	We call $Fil$ the \emph{induced} filtration from $Fil_0$.
	\begin{proof}
		Let $\lambda d: \sO_{X}(D)\to \sO_X(D)\otimes \Omega_{X/k}(\log D)$ be the parabolic $\lambda$-connection induced by the exterior differential $d$. Then for any $t\in \R$, 
		$$
		(E_{t-1},\nabla)\cong (E_t,\nabla)\otimes (\sO_{X}(D),\lambda d).
		$$
		Thus for $t\in \Z$, the Griffiths transversality of $Fil_t$ follows from that of $Fil_0$. Then, for any $t\in \R-\Z$, $Fil_t$ is also Griffiths transverse by construction. Again by construction, for $n\leq s\leq t<n+1$, $\phi_{t,s}$ is injective. Note that the composite
		$$
		Gr_{Fil_{n+1}}(V_{n+1},\nabla)\stackrel{\phi_{n+1,t}}{\longrightarrow} Gr_{Fil_{t}}(V_{t},\nabla)\stackrel{\phi_{t,n}}{\hookrightarrow} Gr_{Fil_{n}}(V_{n},\nabla)
		$$
		is $\phi_{n+1,n}$. As $Fil_0$ is saturated, it follows that $\phi_{n+1,n}$ is injective. So $\phi_{n+1,t}$ is also injective. Then it follows that $\phi_{t,s}$ is injective for any $t\geq s$, by writing it as a composite of finitely many injective morphisms.  
	\end{proof}
	\begin{defn}
		A filtration $Fil$ over $(E_{\bullet},\nabla)$ is \emph{gr-semistable} if it is the induced filtration from a saturated and Griffiths transverse filtration $Fil_0$ of finite level over $(E_0,\nabla)$ such that the associated graded parabolic 0-connection $Gr_{Fil}(E_{\bullet},\nabla)$ is $\mu_L$-semistable.
	\end{defn}
	For $D=\emptyset$, the above notion reduces back to the one considered in Appendix \cite{LSZ19}. We have a natural generalization of \cite[Theorem A.4]{LSZ19} to the parabolic setting.
	\begin{prop}\label{grsemistable}
		Let $k$ be an algebraically closed filed, and $X$ a smooth projective variety over $k$, equipped with a reduced effective normal crossing divisor $D$. Let $L$ be an ample line bundle over $X$. Let $(E_\bullet,\nabla)$ be a $\mu_L$-semistable parabolic $\lambda$-connection over $(X,D)$. Then there exists a gr-semistable filtration $Fil$ over $(E_\bullet,\nabla)$. 
	\end{prop}
	\begin{proof}
		We shall adapt the proof of \cite[Theorem A.4]{LSZ19} to the current setting. One starts with an arbitrary filtration $Fil$, induced from some saturated Griffiths transverse filtration $Fil_0$ of level $n$ over $(E_0,\nabla_0)$. Consider the associated graded parabolic 0-connection 
		$$
		(E_{\bullet}=\bigoplus_{i=0}^nE_{\bullet}^i,\theta):=Gr_{Fil}(E_{\bullet},\nabla).
		$$ 
		If it is semistable, we are done. Otherwise, let $I_{Fil,\bullet}$ be its maximal destabilizer-it is \emph{the} maximal $\theta$-invariant parabolic subsheaf of $E_{\bullet}$ with maximal parabolic slope. It is graded, that is, for $0\leq i\leq n$ and $0\leq t<1$, one has
		$$
		I_{Fil,t}=\bigoplus_{i=0}^{n} I^i_{Fil,t}, \quad I^i_{Fil,t}=I_{Fil,t}\cap E^i_{t}.
		$$
		Moreover, $I_{Fil,t}=I_{Fil,0}\cap E_{t}$ for $0\leq t<1$ and hence 
		$$
		I^i_{Fil,t}=I_{Fil,t}\cap E^i_{t}=I_{Fil,0}\cap E_{t}\cap E^i_{t}=I_{Fil,0}\cap E_t\cap E^i_{0}=I^i_{Fil,0}\cap E^i_t. 
		$$
		The key is the construction of a new filtration $\xi(Fil)$ (of level $n+1$), which is defined by  
		$$
		\xi(Fil)_\bullet^i:=\ker(E_\bullet \longrightarrow  \dfrac{E_\bullet/Fil_\bullet^i}{I_{Fil,\bullet}^{i-1}}),\ 1\leq i\leq n+1, \qquad \xi(Fil)_\bullet^0=E_{\bullet}.
		$$
		Claim: $\xi(Fil)$ is the induced filtration from $\xi(Fil)_0$, which is again a saturated Griffiths transverse filtration over $(E_0,\nabla)$. Note that $\xi(Fil)_0$ is nothing but the $\xi(Fil_0)$ in the terminology of \cite[A.6.2]{LSZ19}, and therefore satuated and Griffiths transverse. It remains to show the induced property. For that, we may assume $0\leq t<1$. Note that there is a commutative diagram:
		$$
		\begin{tikzcd}
			0\arrow{r}& Fil_{t}^i \arrow{d}{} \arrow{r}& \xi(Fil)_{t}^i \arrow{r}\arrow{d} &
			I_{Fil,t}^{i-1} \arrow{r}\arrow{d} & 0\\
			0\arrow{r}&Fil^i_0\cap E_t \arrow{r} &
			\xi(Fil)^i_{0} \cap
			E_{t} \arrow{r} &I^i_{Fil,0}\cap E^i_t \arrow{r}
			& 0
		\end{tikzcd}
		$$
		The top and bottom horizontal sequences are exact. The left and right vertical morphism are identity, and the middle vertical morphism is injective. It follows that the middle vertical morphism is also identity. The claim is proved.  
		
		By the claim, there is a short exact sequence of graded parabolic 0-connections:
		\begin{equation}\label{exact graded}
			0 \longrightarrow Gr_{Fil_{\bullet}}(E_\bullet)/I_{Fil,\bullet} \longrightarrow Gr_{\xi(Fil_{\bullet})_{\bullet}}(E_\bullet)\longrightarrow I_{Fil,\bullet}^{[1]}\longrightarrow 0
		\end{equation}
		Granted the above short sequence, the arguments for \cite[Lemmata A.7-A.8]{LSZ19} work verbatim for the parabolic setting. They imply that, there exists some $k_0\in \N$, such that $\xi^{k_0}(Fil)$ is a gr-semistable filtration. 	 
	\end{proof}

	We proceed to the proof of Theorem \ref{restriction thm}.
	\begin{proof} 
		We observe that the proof of \cite[Proposition 3.29]{MO06} is characteristic free, and does not require components of $D$ be smooth. As remarked in \S1, the integrality of a Higgs field has not been used in the proof of \cite[Lemma 3.7]{S92}. Therefore, we have the extension of Mochizuki's restriction theorem to a semistable parabolic 0-connection over $X_{\log}$.

		Now let $(E_\bullet,\nabla)$ be a semistable parabolic $\lambda$-connection with $\lambda\neq 0$. Take a gr-semistable filtration $Fil_{\bullet}$ on $(E_\bullet,\nabla)$ by Proposition \ref{grsemistable}. By \cite[Proposition 1.5]{MR82}, for a general hyperplane section $Y$, one has
		$$
		Gr_{Fil_{\bullet}|_{Y}}(E_{\bullet},\nabla)|_{Y}\cong (Gr_{Fil_{\bullet}}(E_{\bullet},\nabla))|_{Y}.
		$$
		Note that $Gr_{Fil_{\bullet}}(E_{\bullet},\nabla)$ is a parabolic $0$-connection. By Mochizuki's restriction theorem, $(Gr_{Fil_{\bullet}}(E_{\bullet},\nabla))|_{Y}$ is semistable for a general hyperplane section $Y$ of arbitrarily high degree. Because taking grading preserves the parabolic degree, the semistability of $Gr_{Fil_{\bullet}|_{Y}}(E_{\bullet},\nabla)|_{Y}$ implies the semistability of $(E_{\bullet},\nabla)|_{Y}$, as claimed. 
	\end{proof}

	\section{Tensor product theorem}\label{applications}
	In this section, we shall establish Theorem \ref{main thm}. To begin with, we prove three preparatory lemmata. The first lemma deals with the stability with respect to varying parabolic weights. 
	\begin{lem}\label{par wt}
		Notations as Theorem \ref{main thm}. Let $(E_\bullet,\nabla)$ be a  parabolic $\lambda$-connection over $X_{\log}$ with parabolic weights $\vec{a}=(a_{ij})$. 
		\begin{enumerate}
			\item If $(E_\bullet,\nabla)$ is $\mu_L$-stable with respect to $\vec{a}$, then for any $n=1,2,\dots$, there exists a vector $\vec{b}^n=(b^n_{ij})$ in euclidean space with $b^n_{ij}\in \Q\cap [0,1)$ and $0<|\vec{b}^n-\vec{a}|<\frac{1}{n}$ such that $(E_\bullet,\nabla)$ is $\mu_L$-stable with respect to the parabolic weights $\vec{b^n}$. 	Moreover, if $\textrm{char}\ k>0$, then one may choose the denominator of each $\{b^n_{ij}\}$ to be coprime to $\textrm{char}\ k$.
			\item If there exists a sequence of rational vectors  $\vec{b}^n=(b^n_{ij})$ as above such that  $(E_\bullet,\nabla)$ is $\mu_L$-semistable with respect to the parabolic weights $\vec{b}^n$, then $(E_\bullet,\nabla)$ is $\mu_L$-semistable with respect to $\vec{a}$.
		\end{enumerate}
	\end{lem}
	
	\begin{proof}
		For a vector $\vec{b}=(b_{ij})$ with $b_{ij}\in [0,1)$, $(E_{\bullet},\nabla)$ being $\mu_L$-stable with respect to $\vec{b}$ means that for any $\nabla$-invariant parabolic coherent subsheaf $F_\bullet \subset E_\bullet$, one has the following inequality
		\begin{equation} \label{stable inequality}
			\deg_L(F_0)<\frac{\rk(F_0)}{\rk(E_0)}\deg_L(E_0)+\sum_{i,j} \Big(\frac{\rk(F_0)}{\rk(E_0)}n_{ij}-n_{ij}^{F_\bullet}\Big)\deg_L(D_j) b_{ij}.
		\end{equation}
		Suppose $(E_{\bullet},\nabla)$ to be $\mu_L$-stable with respect to $\vec{a}$. Set 
		$$
		A=\deg_L(E_0)+\sum_{i,j} n_{ij}\deg_L(D_j) a_{ij}.
		$$
		Then the inequality \ref{stable inequality} implies that there is an upper bound $\deg_L(F_0)\leq A$ for all $\nabla$-invariant coherent subsheaves $F_\bullet \subset E_\bullet$. On the other hand, it is clear that there exists a constant $A'$ so that once $\deg_L(F_0)<A'$ is satisfied, the inequality \ref{stable inequality} holds automatically. Hence, the stability condition imposes  only finitely many inequalities for $\vec{b}$. We list them as
		$$
		L_1(\vec{b})> 0, L_2(\vec{b})> 0,\dots, L_m(\vec{b})> 0,
		$$
		where each $L_i$ is linear with \emph{rational} coefficients. In particular, it cuts out an \emph{open} rational cone $C$ in the vector space of $\vec{b}$. Since $C$ contains a solution $\vec{a}$, it must be of positive dimension. Therefore, we may find an approximating sequence $\vec{b}^n$ to $\vec{a}$ with rational entries, as requested in (i). 
		
		Conversely, assume that $(E_\bullet,\nabla)$ is $\mu_L$-semistable with respect to $\vec{b^n}$. Then by the continuity of $L_i$, it follows that $(E_\bullet,\nabla)$ is $\mu_L$-semistable with respect to  $\vec{a}$.

	\end{proof}

	Next we consider the necessary device for mod $p$ reduction. In particular,  we shall assume $\textrm{char}\ k=0$ in the second lemma. We may assume $\lambda\in \{0,1\}$, since $\lambda^{-1}\nabla$ is a connection for $\lambda\neq 0$. For the given tuple $\{X,D,L,E_{\bullet},\nabla\}$ over $k$, there exists an integral (affine) scheme $S$ of finite type over $\Z$, and a new tuple $\{\sX, \sD, \sL, \sE_{\bullet}, \nabla\}$ defined over $S$, where
	\begin{itemize}
		\item [(a)] $\sX$ is smooth projective over $S$,
		\item [(b)] $\sD$ is an $S$-divisor with normal crossing,
		\item [(c)] $\sL$ is an $S$-ample line bundle over $\sX$,
		\item [(d)] $(\sE_{\bullet},\nabla)$ is an $S$-relative parabolic $\lambda$-connection over $(\sX,\sD)$,
	\end{itemize}
	together with a $k$-rational point $s_{\infty}: \Spec\ k\to S$ such that
	$$
	\{\sX, \sD, \sL, \sE_{\bullet}, \nabla\}_{s_{\infty}}:=\{\sX, \sD, \sL, \sE_{\bullet}, \nabla\}\times_{\Spec\ k}S\cong \{X,D,L,E_{\bullet},\nabla\}.
	$$
	Such a tuple is called to be an \emph{arithmetic model} of $\{X,D,L,E_{\bullet},\nabla\}$.
	\begin{lem}\label{mod p red}
		Notations as above. Assume $\dim X=1$.Then $(E_{\bullet},\nabla)$ is $\mu_L$-(semi)stable if and only if there exists infinitely many prime numbers $\{p_i\}$ and geometric points $\{s_i\in S\}$ with $\textrm{char}\ k(s_i)=p_i$ such that
		$(\sE_{\bullet},\nabla)_{s_i}$ is $\mu_{\sL_{s_i}}$-(semi)stable.
	\end{lem}
	\begin{proof}
		The if direction follows from the fact that the Harder-Narasimhan function is upper-semicontinuous (\cite[Theorem 3]{S92}, see also \S2\cite{N09}). The basic idea is that the maximal destabilizer of $(E_{\bullet},\nabla)$ spreads over an open subscheme $S'\subset S$ containing $s_{\infty}$. For $\lambda=0$, the only if direction follows from the openness of (semi)stability. In any event, we provide below a proof of the only if direction for an arbitrary $\lambda$. Assume the contrary. We may shrink $S$ such that for any geometric point $s\in S$, $(\sE_{\bullet},\nabla)_{s}$ is unstable. Then there exists a pair of numbers $(d,r)$, an infinite sequence of prime numbers $\{p_i\}_i$ and a sequence of geometric points $\{s_i\}_i\subset S$ with $\textrm{char}\ k(s_i)=p_i$, such that the degree zero term $F^i_0$ of the minimal quotient $(\sE_{\bullet},\nabla)_{s_i}\twoheadrightarrow(F^i_{\bullet}, \nabla_{s_i})$ is of $\deg(F^i_0)=d$ and $\rk (F^i_{0})=r$. Let $\textrm{Quot}^{(d,r)}_{\sE_0/\sX/S}$ be the Quot-scheme parametrizing quotient coherent sheaves of $(\sE_{0})_s$ whose degree is $d$ and rank is $r$ (see e.g. \cite[Theorem 1.5]{M16}). There is a \emph{closed} subscheme $\textrm{Quot}^{(d,r),\nabla}_{\sE_0/\sX/S}$ parametrizing $\nabla$-invariant quotient coherent sheaves.  To see this, let $Q=\textrm{Quot}^{(d,r)}_{\sE_0/\sX/S}$ and $\pi:\sE_{Q}\twoheadrightarrow \sG$ be the universal quotient. Let $\sK=\ker(\pi)$ and $\nabla_{Q}$ be the induced $\lambda$-connection on $\sE_{Q}$. Consider the composite 
		$$
		\Phi: \sK\stackrel{\nabla_{Q}}{\longrightarrow}\sE_{Q}\otimes \Omega_{\sX_{Q}/Q}(\log \sD_{Q})\stackrel{\pi\otimes \textrm{id}}{\longrightarrow} \sG\otimes \Omega_{\sX_{Q}/Q}(\log \sD_{Q}). 
		$$
		Let $\sZ\to  \sX_{Q}$ be the closed subscheme defined by the zero locus of $\Phi$. Its composite with the natural projection $\sX_{Q}\to Q$ gives the projective morphism $\alpha: \sZ\to Q$. Let $\alpha(\sZ)$ be the scheme theoretical image of $\alpha$, which is a closed subscheme of $Q$. Then $\textrm{Quot}^{(d,r),\nabla}_{\sE_0/\sX/S}\to \alpha(\sZ)$ is the closed subscheme of $\alpha(\sZ)$ (hence also a closed subscheme of $Q$) over which $\alpha$ is of relative dimension one.  Because $(\textrm{Quot}^{(d,r),\nabla}_{\sE_0/\sX/S})_{s_i}\neq \emptyset$ for each $i$, it follows from Hilbert's Nullstellensatz that $\textrm{Quot}^{(d,r),\nabla}_{\sE_0/\sX/S}(k)\neq \emptyset$. But then it means that there is a $\nabla$-invariant quotient sheaf $E_0\twoheadrightarrow F_0$ such that $\mu_L(F_{\bullet})<\mu_L(E_{\bullet})$ where $F_{\bullet}$ is the induced parabolic structure on $F_0$. Contradiction. This shows the only if direction for semistability. The only if direction for stability is proved along the same line.
	\end{proof}

	Finally, we shall need to establish a supplementary property of the generalized BIS correspondence (Proposition \ref{BIS}).  
	\begin{lem} \label{commutative} Notations as above. The parabolic pushforward as well as the parabolic pullback commute with direcet sum and tensor product.
	\end{lem}
	
	\begin{proof}
		It suffices to show the parabolic pushforward commutes with direct sum and tensor product by the equivalence of categories. Let $\pi: X'\to X$ be as above. Recall that for a $G$-equivariant sheaf $V$ over $X'$, the parabolic structure on $\pi_{\mathrm{par}*}V$ is given by 
		$$
		(\pi_{\mathrm{par}*}V)_t:=(\pi_*(V\otimes \sO_{X'}(\sum_j[-t N]D_j'))^{G}.
		$$
		Since direct sum commutes with tensor product, pushforward $\pi_*$ and taking $G$-invariants, it follows that parabolic pushforward commutes with direct sum. Let $(V_i,\nabla_i),i=1,2$ be two $G$-equivariant $\lambda$-connections over $X'_{\log}$. One notices that there is a natural morphism of parabolic $\lambda$-connections
		$$
		\varphi: \pi_{\mathrm{par}*}(V_1,\nabla_1)\otimes \pi_{\mathrm{par}*}(V_2,\nabla_2)\to \pi_{\mathrm{par}*}((V_1,\nabla_1)\otimes (V_2,\nabla_2)).
		$$
		We are going to show that $\varphi$ is an isomorphism. The problem is local. As $\dim X=1$, by Lemma \ref{parabolic vb over curve} it follows that, we may assume that $V_i, i=1,2$ are direct sums of $G$-equivariant line bundles (via the BIS correspondence, direct sums of parabolic line bundles over $X_{\log}$ and direct sums of $G$-equivariant line bundles over $X'$ correspond to each other.). We may assume further that $V_i$s are simply $G$-equivariant line bundles. 
		
		Etale locally we may assume that $$X'=\{y^N=x\}\subset \A^2_k\stackrel{pr}{\to} X=\{y=0\}\cong \A^1,$$ and $V_i=\mathcal{O}_{X'}e_i$ is the trivial line bundles with base element $e_i$. Take a generator $\sigma\in G$.  Then the $G$-action on $\sO_{X'}$ and $V_i$ are described respectively by 
		$$
		\sigma(y)=\xi y, \quad \sigma(e_i)=\xi^{-n_i} e_i,
		$$ 
		where $\xi$ is an $N$-th primitive root of unity, and $0\leq n_i/N<1$ is the parabolic weight of $\pi_{\mathrm{par}*} V_i$. Then the paraboloic structure of $\pi_{\mathrm{par}*} V_i$ is determined by
		\begin{equation*} 
			(\pi_{\mathrm{par}*} V_i)_t=\left \{
			\begin{array}{lcl}
				\mathcal O_Xy^{n_i}e_i & & {0\leq t\leq \frac{n_i}{N}}\\
				\mathcal O_Xxy^{n_i}e_i & & { \frac{n_i}{N}<t<1}.
			\end{array} \right.
		\end{equation*}
		It follows that the paraboloic structure of $\pi_{\mathrm{par}*}V_1\otimes\pi_{\mathrm{par}*}V_2$ is determined by
		\begin{equation*} 
			(\pi_{\mathrm{par}*}V_1\otimes\pi_{\mathrm{par}*}V_2)_t=\left \{
			\begin{array}{lcl}
				\mathcal O_Xy^me_1\otimes e_2 & & {0\leq t\leq \frac{m}{N}}\\
				\mathcal O_Xxy^me_1\otimes e_2 & & { \frac{m}{N}<t<1}.
			\end{array} \right.
		\end{equation*}
		Here  $m=n_1+n_2$ if $n_1+n_2 < N$, and $m=n_1+n_2-N$ if $n_1+n_2 \geq N$. On the other hand, $V_1\otimes V_2$ is the trivial bundle with base element $e_1\otimes e_2$. It is straightforward to verify that the last expression is nothing but the parabolic structure of $\pi_{\mathrm{par}*}(V_1\otimes V_2)$ for $0\leq t<1$. Therefore $\varphi$ is indeed an isomorphism. The lemma is proved.
	\end{proof}

	Now we proceed with the proof of Theorem \ref{main thm}.
	\begin{proof}
		By Theorem \ref{restriction thm}, we may assume $\dim X=1$ in the following. We divide the whole proof into five steps. In the first four steps, we shall assume $(E^i_{\bullet},\nabla^i), i=1,2$ to be $\mu_L$-stable.
		
		{\itshape Step 1.} Let $X$ be a smooth projective curve defined over $k$, and $D\subset X$ a reduced effective divisor. By Lemma \ref{par wt}, we may assume the parabolic weights of $E^i_{\bullet}, i=1,2$ are contained in $\frac{1}{N}\Z$ and $(N,p)=1$ when $\textrm{char}\ k=p>0$. Fix a closed point $x_0\in X$ away from $D$. We may choose another reduced effective divisor $B$ such that the degree of $\tilde D=D+B$ is divisible by $N$. Let $\pi: X'\to X$ be the cyclic cover with branch divisor $\tilde D$. We claim that the parabolic pullback $\pi_{\mathrm{par}}^*(E^i_{\bullet},\nabla^i)$, which is a $G$-equivariant logarithmic $\lambda$-connection over $(X',\pi^{-1}\tilde D)$, is semistable. Indeed, if it were unstable, then the maximal destabilizer would be $G$-equivariant and invariant under $\pi_{\mathrm{par}}^*\nabla^i$. By the generalized BIS correspondence, it descends to a parabolic subsheaf of $E^i_{\bullet}$ which is $\nabla^i$-invariant. The claim follows. Moreover by Lemma \ref{commutative}, we have 
		$$
		\pi_{\mathrm{par}}^*((E^1_{\bullet},\nabla^1)\otimes (E^2_{\bullet},\nabla^2))\cong \pi_{\mathrm{par}}^*(E^1_{\bullet},\nabla^1)\otimes \pi_{\mathrm{par}}^*(E^2_{\bullet},\nabla^2).
		$$
		Therefore, we may assume the parabolic structure of $E^i_{\bullet}$ to be trivial (Example \ref{example} (a)).\\
		
		{\itshape Step 2.} When $\textrm{char}\ k=0$, we take an arithmetic model of $\{X,D,L,E^1_{\bullet}\oplus E^2_{\bullet},\nabla^1\oplus \nabla^2\}$. By Lemma \ref{mod p red}, there exists an infinite sequence of geometric points $\{s_i\}\subset S$ such that both $(\sE^1_{\bullet},\nabla)_{s_i}$ and $(\sE^2_{\bullet},\nabla)_{s_i}$ are stable. We remove those (at most finitely many) $s_i$s satisfying $p_i<\rk(E^1)+\rk(E^2)-1$ from the sequence. By Lemma \ref{mod p red} again, the char zero statement follows from the char $p$ statement.\\

		{\itshape Step 3.} We choose an arbitrary $W_2(k)$-lifting of the pair $(X,D)$, and let $C^{-1}$ be inverse Cartier transform of Ogus-Vologodsky with respect to this lifting (see Appendix \cite{LSYZ19}). Let $(E^i,\nabla^i), i=1,2$ be two stable logarithmic $\lambda$-connections over $(X,D)$. Then by Proposition \ref{grsemistable}, there exists a filtration $Fil^i_{-1}$ over $E^i$ such that $\Gr_{Fil^i_{-1}}(E^i,\nabla^i)$ is a graded semistable logarithmic Higgs bundle over $(X,D)$. Because of the rank assumption, one has $\rk \Gr_{Fil^i_{-1}}(E^i,\nabla^i)\leq p$. Therefore, there exists a semistable Higgs-de Rham flow initializing $\Gr_{Fil^i_{-1}}(E^i,\nabla^i)$ for $i=1,2$ (\cite[Theorem A.1]{LSZ19}, \cite[Theorem 5.12]{LA13}). That is, we obtain a flow of the following form:
		$$
		\xymatrix@C=4.5ex{
			(E^{i},\nabla^i)\ar[dr]^{Gr_{Fil^i_{-1}}} && (H^i_0,\nabla^i_0)\ar[dr]^{Gr_{Fil^i_0}} &&  (H^i_1,\nabla^i_1)\ar[dr]^{Gr_{Fil^i_1}} \\
			&(E^i_0,\theta^i_0) \ar[ur]^{C^{-1}} & & (E^i_1,\theta^i_1) \ar[ur]^{C^{-1}}&&
			\cdots,
		}
		$$
		in which each Higgs term in bottom is semistable. Since a Higgs-de Rham flow preserves the rank, the rank assumption implies that $\rk E^1_j+\rk E^2_j\leq p+1$ for each $j\geq 0$. So the nilpotent indices $l^i_j$s of $\theta^i_j$s satisfy $l^1_j+l^2_j\leq p-1$. By Lemma 3.4 \cite{KS20}, it follows that
		$$
		C^{-1}((E^1_j,\theta^1_j)\otimes (E^2_j,\theta^2_j))\cong (H^1_j,\nabla^2_j)\otimes (H^2_j,\nabla^2_j).
		$$
		Thus, we obtain the tensor product of flows:
		$$
		\xymatrix@C=2.5ex@R=8ex{
			(E^{1},\nabla^1)\otimes (E^2,\nabla^2)\ar[dr]^-{Gr_{Fil^1_{-1}\otimes Fil^2_{-1}}} && (H^1_0,\nabla^2_0)\otimes (H^2_0,\nabla^2_0)\ar[dr]^-{Gr_{Fil^1_{0}\otimes Fil^2_{0}}} & \\
			&(E^1_0,\theta^1_0)\otimes (E^2_0,\theta^2_0) \ar[ur]^{C^{-1}} & &
			\cdots.
		}
		$$
		In the next step, we are going to show the above flow is semistable, that is, each Higgs term in bottom is semistable. In particular, $(E^1_0,\theta^1_0)\otimes (E^2_0,\theta^2_0)$ is semistable. It follows that $(E^{1},\nabla^1)\otimes (E^2,\nabla^2)$ is semistable, as desired.\\

		{\itshape Step 4.} 
		Set $r_i=\rk(E^i)$ and $d_i=\deg(E^i)$. In case $r_i|d_i$, one may twist $(E^i_0,\theta^i_0)$ with some invertible sheaf of degree $-d_i/r_i$. The resulting logarithmic Higgs bundle is still semistable and of degree zero. As $r_1+r_2\leq p+1$, we may assume that $p\geq 3$ and $r_i\leq p-1$ for each $i$. After taking the base change via a certain cyclic cover of degree $r_i$, one may reduce the problem to the situation that $r_i|d_i$. Therefore, we may assume $\deg E^i_0=0$ and consequently $\deg E^i_j=0$ for $i=1,2$ and $j\geq 0$. Since the ranks and the degrees of $\{E_j^i\}_j$ are constant for $i=1,2$, they form \emph{bounded} families. It follows that $\{E^1_j\otimes E^2_j\}_j$ is also a bounded family. In particular, the degrees of subsheaves of $\{E^1_j\otimes E^2_j\}_j$ have a upper bound by a theorem of Grothendieck (see e.g. \cite[Proposition 3.12]{M16}). But then, if there were a $\theta$-invariant subsheaf $F\subset E^1_0\otimes E^2_0$ of positive degree, then by the flow it would give rise to a $\theta$-invariant subsheaf of $E^1_j\otimes E^2_j$, whose degree exceeds the upper bound when $j$ becomes large enough. Contradiction.\\

		{\itshape Step 5.} The final step reduces the general case to the previous one by using a Jordan-H\"older (JH) filtration attached to a semistable parabolic $\lambda$-connection. Let 
		$$0=(E_{\bullet}^{i,0},\nabla^i)\subset (E_{\bullet}^{i,1},\nabla^i)\subset\dots\subset (E_{\bullet}^{i,m_i},\nabla^i)=(E_{\bullet}^i,\nabla^i),$$
		be a JH filtration of $(E_{\bullet}^i,\nabla^i), i=1,2$, that is, each subquotient $$\Gr_{E_{\bullet}^i}^j:=(E_{\bullet}^{i,j}/E_{\bullet}^{i,j-1},\nabla^i),\quad i=1,2, \quad 1\leq j\leq m_i$$
		is $\mu_L$-stable of slope $\mu_L(E_{\bullet}^i)$. Then we define a filtration of parabolic $\lambda$-connections over the tensor product as follows: 
		$$0=(F^0_\bullet,\nabla)\subset (F^1_\bullet,\nabla)\subset \dots \subset (F^{m_1m_2}_\bullet,\nabla)=(E^{1}_\bullet,\nabla^1)\otimes (E^2_\bullet,\nabla^2),$$
		where for $0\leq i\leq m_1-1$ and $1\leq j\leq m_2$
		$$
		F^{im_2+j}_{\bullet}=E^{1,i}_{\bullet}\otimes E^2_{\bullet}+E^{1,i+1}_{\bullet}\otimes E^{2,j}_{\bullet}.
		$$
		
		One verifies that 
		$$
		(F^j_\bullet/F^{j-1}_\bullet,\nabla)\cong \Gr_{E_{\bullet}^1}^{\left\lceil \frac{j}{m_2} \right\rceil}\otimes  \Gr_{E_{\bullet}^2}^{j-\big(\left\lceil \frac{j}{m_2} \right\rceil-1\big)m_2}.
		$$
		By the previous discussion, $(F^j_\bullet/F^{j-1}_\bullet,\nabla)$ is $\mu_L$-semistable of slope
		\begin{equation*}
			\begin{split}
				\mu_L\bigg(\Gr_{E_{\bullet}^1}^{\left\lceil \frac{j}{m_2} \right\rceil}\otimes \Gr_{E_{\bullet}^2}^{j-\big(\left\lceil \frac{j}{m_2} \right\rceil-1\big)m_2} \bigg)&=\mu_L\bigg(\Gr_{E_{\bullet}^1}^{\left\lceil \frac{j}{m_2} \right\rceil}\bigg)+\mu_L\bigg(\Gr_{E_{\bullet}^2}^{j-\big(\left\lceil \frac{j}{m_2} \right\rceil-1\big)m_2}\bigg)\\
				&=\mu_L(E^1_\bullet)+\mu_L(E^2_\bullet),
			\end{split}
		\end{equation*}
		where the first equality follows from Lemma \ref{tensorslope}. It follows that $ (E^{1}_\bullet,\nabla^1)\otimes (E^2_\bullet,\nabla^2)$ is $\mu_L$-semistable.
	\end{proof}

	\appendix

	{\bf Acknowledgements.} We are grateful to H\'{e}l\`{e}ne Esnault for many helpful comments and valuable suggestions. Especially, she pointed out to us an error in our early argument about the Mehta-Ramanathan type theorem for parabolic $\lambda$-connections (see footnote 2) as well as an error in Lemma \ref{induced filtration}.

	\bibliographystyle{plain}

\end{document}